\documentclass[11pt,fullpage]{article}
\usepackage{amsmath}
\usepackage{amsfonts,amssymb}
\usepackage{amsthm}
\usepackage{graphics,graphicx}
\usepackage{hyperref}
\usepackage[right = 2.5cm, left=2.5cm, top = 2.5cm, bottom =2.5cm]{geometry}
\newcommand{\Real}{\mathbb R}
\newcommand{\norm}[1]{\|#1\|}
\newcommand{\abs}[1]{\left\vert#1\right\vert}
\newcommand{\set}[1]{\left\{#1\right\}}

\newcommand{\grad}{\nabla}

\newcommand{\conv}[2]{#1 \ast #2}

\newcommand{\K}{\mathcal{K}}

\newcommand{\wkspace}[1]{L^{#1,\infty}}

\newtheorem{theorem}{Theorem}
\theoremstyle{remark}
\newtheorem{remark}{Remark}

\theoremstyle{lemma}

\theoremstyle{definition}
\newtheorem{definition}{Definition}
\theoremstyle{lemma}
\newtheorem{lemma}{Lemma}

\begin{document}

\title{Intermediate Asymptotics for Critical and Supercritical Aggregation Equations and Patlak-Keller-Segel models} 

\author{Jacob Bedrossian \footnote{\textit{jacob.bedrossian@math.ucla.edu}, University of California-Los Angeles, Department of Mathematics}}

\date{}

\maketitle

\begin{abstract} 
We examine the long-term asymptotic behavior of dissipating solutions to aggregation equations and Patlak-Keller-Segel models with degenerate power-law and linear diffusion. 
The purpose of this work is to identify when solutions decay to the self-similar spreading solutions of the homogeneous
diffusion equations. 
Combined with strong decay estimates, entropy-entropy dissipation methods provide a natural solution to this question and make it possible to derive quantitative convergence rates in $L^1$.
The estimated rate depends only on the nonlinearity of the diffusion and the strength of the interaction kernel at long range.  
\end{abstract}

\section{Introduction} 
The most widely studied mathematical models of nonlocal aggregation phenomena are the Patlak-Keller-Segel (PKS) models, originally introduced to study the chemotaxis of microorganisms \cite{Patlak,KS,Hortsmann,HandP}. Similar models are also used to study the formation of herds and flocks in ecological systems \cite{Bio,Topaz,Milewski,GurtinMcCamy77}.
A common theme is the competition between the tendency for organisms to diffuse, e.g. under Brownian motion or to avoid over-crowding, 
and for organisms to aggregate into groups through nonlocal self-attraction.   
The parabolic-elliptic PKS models are a subclass of the general aggregation-diffusion equations
\begin{equation}
u_t + \grad \cdot (u \grad \K \ast u) = \Delta A(u). \label{def:ADD_general}
\end{equation}
The local and global existence and uniqueness of models such as \eqref{def:ADD_general} is well studied (see for instance \cite{BRB10,BertozziSlepcev10,Blanchet09,BlanchetEJDE06,SugiyamaDIE06,SugiyamaADE07,SugiyamaDIE07,Corrias04}).
However, less is known about the long-term qualitative behavior of solutions. 
In this work,  we are interested in examining the asymptotic profiles of dissipating solutions to \eqref{def:ADD_general} in the special case
\begin{equation} \label{def:ADD}
\left\{
\begin{array}{l}
  u_t + \grad \cdot (u \grad \K \ast u) = \Delta u^m, \;\; m \geq 1, \\ 
  u(0,x) = u_0(x) \in L_+^1(\Real^d;(1+\abs{x}^2)dx)\cap L^\infty(\Real^d),
\end{array}
\right.
\end{equation}
where  $L_+^1(\Real^d;\mu) := \set{f \in L^1(\Real^d;\mu): f \geq 0}$. 
In particular, we are interested in determining when solutions to \eqref{def:ADD} converge in $L^1(\Real^d)$ as $t \rightarrow \infty$ to the self-similar spreading solutions of the diffusion equation 
\begin{equation}
u_t = \Delta u^m \label{def:PME}.
\end{equation}
All dissipating solutions are weak$^\star$ converging to zero as $t \rightarrow \infty$, but this kind of result implies that for $1 << t < \infty$, the dissipating solutions all look more or less like self-similar solutions of \eqref{def:PME}. For this reason, these results are often referred to as \emph{intermediate asymptotics}. 

\emph{Supercritical} problems are those in which the aggregation is dominant at high concentrations, \emph{subcritical} problems are those in which the diffusion dominates at high concentrations, and \emph{critical} problems are those in which the effects are in approximate balance.
It is known that supercritical problems exhibit finite time blow up for solutions of arbitrarily small mass and subcritical problems have global solutions \cite{SugiyamaDIE06,SugiyamaADE07,BRB10,Blanchet09}.   
The critical case is more interesting; data with small mass exists globally, whereas finite time blow up is possible for large mass \cite{Blanchet09,BRB10,BlanchetEJDE06,SugiyamaADE07}.
In this work, we will refer to the case $m < 2-2/d$ as supercritical and $m = 2-2/d$ as critical. This is in contrast to the definition used in \cite{BRB10}, where the critical diffusion exponent was taken to depend on the singularity of the kernel. 
Here, achieving such a precise balance is not the primary interest and moreover we are concerned with examining the limit of low concentrations.
In the sense of \cite{BRB10,Blanchet09,SugiyamaADE07,SugiyamaDIE06}, $m = 2-2/d$ is the
critical exponent for the Newtonian potential, which is the most singular kernel known to have unique, local-in-time solutions \cite{BRB10}.

As strong nonlinearities vanish quickly near zero, scaling heuristics suggest that the nonlocal aggregation term should become irrelevant for small data in the critical and supercritical regime.
We use entropy-entropy dissipation methods \cite{CarrilloToscani98,Toscani99,CarrilloToscani00,CarrilloEntDiss01,CarrilloDiFranToscani06,BilerDolbeaultEsteban02} to obtain several intermediate asymptotics results which show this to be true, and that solutions of \eqref{def:ADD} converge to self-similar solutions of \eqref{def:PME}.
Entropy-entropy dissipation methods are well-suited for proving the convergence to equilibrium states of nonlinear Fokker-Plank-type equations for arbitrary data \cite{CarrilloToscani98,CarrilloEntDiss01}.  
Through a change of variables employed below, this also provides convergence to self-similarity of nonlinear homogeneous diffusion equations \cite{CarrilloToscani00}. 
In contrast to these works, we employ such methods to prove a \emph{small data} result, treating the nonlocal aggregation term as a perturbation. 
For this to work, sufficiently strong decay estimates on the solution must be obtained. 
Indeed, strong decay estimates imply the intermediate asymptotics results, and so we have chosen to state them separately in Theorem \ref{thm:Decay} below.   
Here, we obtain these estimates using iteration methods, discussed in more detail below, which are a refinement of the local theory of \eqref{def:ADD} (see e.g. \cite{Blanchet09,BRB10}).
While nonlinear, they are essentially perturbative in nature and thus somewhat limited against arbitrary data, using basic dissipation estimates to over-power the nonlocal advection term only under certain conditions. Analogous to related models, such as the nonlinear Schr\"odinger equations, it is likely a fully non-perturbative theory will need to be applied in order to treat large data, which is sometimes significantly more difficult (see for instance \cite{Tao,KillipVisanClay}).    
More details and discussion about the results and the methods of the proofs are discussed below in \S\ref{sec:StatResults} and \S\ref{sec:outline}.   

The first of our intermediate asymptotics results, Theorem \ref{thm:IA}, covers the case $\K \in W^{1,1}(\Real^d)$. 
Here, the nonlocal term can be considered to have a finite characteristic length-scale which becomes vanishingly small relative to the length-scale of the solution as it dissipates. 
A result similar to Theorem \ref{thm:IA} for $L^p, \; 1 < p < \infty$, was proved for the special case of the Bessel potential in \cite{LuckhausSugiyama06,LuckhausSugiyama07} with the soft compactness method of \cite{KaminVazquez88} (see also \cite{VazquezPME}).
In contrast to methods based on compactness, the entropy-entropy dissipation methods obtain quantitative convergence rates in $L^1$,
which by interpolation against the decay estimates, provides convergence in all $L^p$, $1 \leq p < \infty$. For supercritical problems, the convergence rate is shown to be the same as the optimal rates for \eqref{def:PME} \cite{CarrilloToscani98,Toscani99,CarrilloToscani00,CarrilloEntDiss01,VazquezPME}. 

In general, if the kernel does not have critical scaling at large length-scales, the long-range effects should still become irrelevant as the solution dissipates.
That is, we should expect results similar to the $\K \in W^{1,1}(\Real^d)$ case to hold, except when $m = 2-2/d$ and  $\grad \K \sim \abs{x}^{1-d}$ as $\abs{x} \rightarrow \infty$. Indeed, when $\K$ is the Newtonian potential, there exists at least one self-similar spreading solution to \eqref{def:ADD} when $m = 2-2/d$ \cite{BlanchetEJDE06,Blanchet09,BlanchetDEF10,CalvezCarrillo10}. In the presence of linear diffusion, these are additionally known to be the global attractors \cite{BlanchetEJDE06,BlanchetDEF10}. 
Theorem \ref{thm:IA2} below extends Theorem \ref{thm:IA} to the general case of $\K \not\in W^{1,1}(\Real^d)$, where the decay of $\K$ is characterized by $\gamma \in [d-1,d]$ such that $\abs{\grad \K(x)} = \mathcal{O}(\abs{x}^{-\gamma})$ as $\abs{x} \rightarrow \infty$. 
We show that if $\gamma > d-1$, then dissipating solutions converge to the self-similar spreading solutions of \eqref{def:PME}. 
However, in contrast to Theorem \ref{thm:IA}, the long-range effects appear to degrade the convergence rate and Theorem \ref{thm:IA2} provides a quantitative estimate of this effect in terms of $m$ and $\gamma$.
It is not known whether the rates obtained in Theorem \ref{thm:IA2} are sharp.
When $\gamma = d-1$, the kernel behaves like the Newtonian potential on large length-scales, and the result is no longer expected to hold if $m = 2-2/d$.
Indeed, we expect solutions to converge to the self-similar solutions of \eqref{def:ADD} constructed in \cite{BlanchetEJDE06,Blanchet09}. 
However, Theorem \ref{thm:IA2} asserts that in supercritical cases, self-similar solutions to \eqref{def:PME} again govern the intermediate asymptotics.
Thus, Theorem \ref{thm:IA2} provides intermediate asymptotics for Patlak-Keller-Segel models with linear diffusion in dimensions $d \geq 3$. 

As remarked above, the results of Theorems \ref{thm:IA} and \ref{thm:IA2} hold because for small data, the diffusion dominates the global dynamics. 
On the other hand, for subcritical problems, the aggregation can dominate at large length-scales, leading to the existence of nontrivial stationary solutions \cite{LionsCC84,BedrossianGlobalMin10}, which clearly violate strong decay estimates such as those stated below in Theorem \ref{thm:Decay}. 
However, between the results here and the work of \cite{LionsCC84,BedrossianGlobalMin10}, not every case is covered.
For instance, if $\K \in L^1(\Real^d)$ and $2-2/d < m < 2$, stationary solutions are only known to exist for sufficiently large mass, and the behavior of smaller solutions is unknown.
Moreover, convergence to these stationary solutions is only known in certain cases \cite{KimYao11}. 

In what follows, we denote $\norm{u}_p := \norm{u}_{L^p(\Real^d)}$ where $L^p(\Real^d) := L^p$ is the standard Lebesgue space. 
We will often suppress the dependencies of functions on space and/or time to enhance readability. 
The standard characteristic function for some $S \subset \Real^d$ is denoted $\mathbf{1}_{S}$ and we denote the ball $B_R(x_0) := \set{x \in \Real^d : \abs{x - x_0} < R}$.
In formulas we use the notation $C(p,k,M,..)$ to denote a generic constant, which may be different from line to line or term to term in the same formula. 
In general, these constants will depend on more parameters than those listed, for instance those which are fixed by the problem, such as $\K$ and the dimension, but these
dependencies are suppressed.
We use the notation $f \lesssim_{p,k,...} g$ to denote $f \leq C(p,k,..)g$ where again, dependencies that are not relevant are suppressed. 

\subsection{Statement of Results} \label{sec:StatResults}
We need the following definition from \cite{BRB10}, which we restate here. 
\begin{definition}[Admissible Kernel] \label{def:adm}
 We say a kernel $\K \in C^3 \setminus \set{0}$ is \emph{admissible} if $\K \in W^{1,1}_{loc}(\Real^d)$ and the following holds:
\begin{itemize}
\item[\textbf{(KN)}] $\K$ is radially symmetric, $\K(x) = k(\abs{x})$ and $k(\abs{x})$ is non-increasing.
\item[\textbf{(MN)}] $k^{\prime\prime}(r)$ and $k^\prime(r)/r$ are monotone on $r \in (0,\delta)$ for some $\delta > 0$. 
\item[\textbf{(BD)}] $\abs{D^3\K(x)} \lesssim \abs{x}^{-d-1}$. 
\end{itemize}
\end{definition}
The definition ensures that the kernel is radially symmetric, attractive, reasonably well-behaved at the origin and has second derivatives which define bounded convolution operators on $L^p$ for $1 < p < \infty$.
It is important to note that all admissible kernels satisfy $\grad \K \in \wkspace{\frac{d}{d-1}}$, where $\wkspace{p}$ denotes the weak-$L^p$ space, making the Newtonian potential effectively the most singular of admissible kernels \cite{BRB10}.
Provided $\K$ is admissible, for a given initial condition $u_0(x) \in L_+^1(\Real^d; (1 + \abs{x}^2)dx ) \cap L^\infty(\Real^d)$, 
\eqref{def:ADD} has a unique, local-in-time weak solution which satisfies $u(t) \in C([0,T);L_+^1(\Real^d;(1+\abs{x}^2)dx) )\cap L^\infty((0,T)\times\Real^d)$ \cite{BRB10,BertozziSlepcev10,BlanchetEJDE06,SugiyamaDIE07,BRB_2D_10}.
Moreover, $u(t)$ is a solution to \eqref{def:ADD} in a sense which is stronger than a distribution solution, which is important for obtaining a well-behaved local theory, but the distinction will not be important here \cite{BRB10,BertozziSlepcev10}. 
Weak solutions conserve mass and we define $M = \norm{u_0}_1 = \norm{u(t)}_1$.   

The self-similar solutions to the diffusion equation \eqref{def:PME} are well-known, see for instance \cite{VazquezPME} and \cite{CarrilloToscani00}.
In the linear case $m=1$, the self-similar solution is simply the heat kernel,
\begin{equation}
\mathcal{U}(t,x;M) = \frac{M}{(4\pi t)^{d/2}}e^{\frac{-\abs{x}^2}{4t}}. \label{def:Ulin} 
\end{equation}
In the case of degenerate diffusion $m > 1$, the self-similar solution is given by the Barenblatt solution,
\begin{equation}
\mathcal{U}(t,x;M) = t^{-\beta d}\left( C_1 - \frac{(m-1)\beta}{2m}\abs{x}^2t^{-2\beta}\right)_+^{\frac{1}{m-1}}, \label{def:UnonLin}
\end{equation}
where $C_1$ is determined from the conservation of mass and
\begin{equation}
\beta = \frac{1}{d(m-1) + 2}. \label{def:beta}
\end{equation}
Note that,
\begin{equation*}
\norm{\mathcal{U}(t;M)}_p \lesssim t^{-d\beta\left(1 - \frac{1}{p}\right)},
\end{equation*}
and so to provide a meaningful characterization of the convergence to self-similarity, quantitative estimates 
will be stated in terms of this relative scale.

The entropy-entropy dissipation methods of \cite{CarrilloToscani98,Toscani99,CarrilloToscani00,CarrilloEntDiss01}
were used to determine the optimal rate of convergence in $L^1(\Real^d)$ to self-similarity. 
That is, any solution $u(t)$ of \eqref{def:PME} satisfies
\begin{equation*}
t^{d\beta\left(1 - \frac{1}{p}\right)}\norm{u(t) - \mathcal{U}(t;M)}_p \lesssim (1+t)^{-\frac{2\beta}{p}\min\left(\frac{1}{2},\frac{1}{m}\right)}, \; \forall p, \, 1 \leq p < \infty.
\end{equation*}
This rate should be contrasted with the rates obtained in Theorems \ref{thm:IA} and \ref{thm:IA2}, where it is shown that kernels with finite length-scales do not have much effect on the rate, but strong nonlocal effects do.

In order to emphasize the relationship between decay estimates and intermediate asymptotics, we state them separately. 
Results similar to \textit{(i)} of Theorem \ref{thm:Decay} have been obtained in a variety of places, for example \cite{PerthameVasseur10,SugiyamaADE07,BlanchetDEF10}. 
Our estimates are obtained in a closely related but different way than existing work. We first rescale into the self-similar variables of the diffusion equation as in \cite{BlanchetDEF10}, and then adapt the Alikakos \cite{Alikakos} iteration techniques of \cite{Kowalczyk05,CalvezCarrillo06,Blanchet09,BRB10}, which are variants of fairly standard methods for obtaining uniform in time $L^\infty$ bounds for PKS models (see also \cite{JagerLuckhaus92,SugiyamaDIE06,SugiyamaADE07,SugiyamaDIE07}).
This approach to decay estimates has the advantage of naturally extending the existing methods used to obtain uniform bounds, and for a relatively mild increase in complexity, much stronger results are obtained. 
Here we use this advantage to also deduce a sufficient condition for decay estimates to hold in the critical case $2-2/d$, \textit{(ii)} of Theorem \ref{thm:Decay} below. 
For critical problems, uniform equi-integrability in time is known to be equivalent to global uniform boundedness for solutions to \eqref{def:ADD} \cite{CalvezCarrillo06,Blanchet09,BRB10}, and due to the similarities in the proof, we may state something analogous for decay estimates. 
Indeed, \eqref{def:EquiInTheta} is simply the requirement that the solution of the rescaled system remain uniformly equi-integrable. 
The proofs of Theorems \ref{thm:Decay},\ref{thm:IA} and \ref{thm:IA2} are outlined in more detail in \S\ref{sec:outline}. 
Remarks on the limitations and possible extensions are made after the statements. 

\begin{theorem}[Decay Estimates] \label{thm:Decay}
Let $d \geq 2$, $m \in [1,2-2/d]$ and $\K$ admissible. 
Let $u_0 \in L^1_+(\Real^d;(1+\abs{x}^2)dx)\cap L^\infty(\Real^d)$. 
\begin{itemize}
\item[(i)] There exists an $\epsilon_0 > 0$ (independent of $u_0$) such that if $\norm{u_0}_1 + \norm{u_0}_{(2-m)d/2} < \epsilon_0$, then the weak solution $u(t)$ to \eqref{def:ADD} which satisfies $u(0) = u_0$ is global and satisfies the decay estimate
\begin{equation}
\norm{u(t)}_\infty \lesssim (1+t)^{-d\beta}. \label{ineq:LinftyDecay} 
\end{equation}
\item[(ii)] If $m = 2-2/d$ and $u(t)$ is a global weak solution to \eqref{def:ADD}  which satisfies
\begin{equation}
\lim_{k \rightarrow \infty}\sup_{t \in [0,\infty)}\int \left(u(t,x) - k\left(\frac{t}{\beta}+1\right)^{d\beta}\right)_+ dx = 0, \label{def:EquiInTheta} 
\end{equation}
then $u(t)$ satisfies the decay estimate \eqref{ineq:LinftyDecay}.    
\end{itemize}
\end{theorem}

Once the decay estimate \eqref{ineq:LinftyDecay} has been established, entropy-entropy dissipation methods can be adapted to deduce the following intermediate asymptotics theorems, as the decay estimate provides sufficient control of the nonlocal terms. 

\begin{theorem}[Intermediate Asymptotics I: Finite Length-Scale] \label{thm:IA} 
Let $d \geq 2$, $m \in [1,2 - 2/d]$ and $\K \in W^{1,1}$ be admissible. 
Suppose $u(t)$ is a global weak solution of \eqref{def:ADD} which satisfies the decay estimate \eqref{ineq:LinftyDecay}.  
If $m < 2-2/d$, then $u(t)$ satisfies
\begin{equation}
t^{d\beta\left(1 - \frac{1}{p}\right)}\norm{u(t) - \mathcal{U}(t;M)}_p \lesssim (1+t)^{-\frac{\beta}{p}}, \; \forall p, \, 1 \leq p < \infty,  \label{ineq:conv_supercrit_IA}
\end{equation}  
and if $m = 2-2/d$, then for all $\delta > 0$, $u(t)$ satisfies
\begin{equation}
t^{d\beta\left(1 - \frac{1}{p}\right)}\norm{u(t) - \mathcal{U}(t;M)}_p \lesssim_{\delta} (1+t)^{-\frac{\beta}{p}(1-\delta)}, \;\; \forall p, \, 1 \leq p < \infty.  \label{ineq:conv_IA}
\end{equation}  
Here $\beta$ is defined in \eqref{def:beta} 
and $\mathcal{U}(x,t;M)$ is the self-similar solution to \eqref{def:PME} with mass $M = \norm{u_0}_1$ given in \eqref{def:Ulin} or \eqref{def:UnonLin}. 
\end{theorem}

\begin{theorem}[Intermediate Asymptotics II: Infinite Length-Scales] \label{thm:IA2}
Let $d \geq 2$ and $\K$ be admissible with $\grad\K(x) = \mathcal{O}(\abs{x}^{-\gamma})$ as $\abs{x} \rightarrow \infty$ for some $\gamma \in [d-1,d]$. If $\gamma = d-1$ then suppose $m  \in [1,2-2/d)$ and otherwise we may take $m \in [1,2-2/d]$.
Suppose $u(t)$ is a global weak solution of \eqref{def:ADD} which satisfies the decay estimate \eqref{ineq:LinftyDecay}.  
Then, for all $\delta>0$,  $u(t)$ satisfies
\begin{equation}
t^{d\beta\left(1 - \frac{1}{p}\right)}\norm{u(t) - \mathcal{U}(t;M)}_p \lesssim_{\delta} (1+t)^{-\frac{\beta}{p}\min\left(1, 1+\gamma - \beta^{-1} - \delta \right)}, \;\; \forall p, \, 1 \leq p < \infty. \label{ineq:conv_IA2}
\end{equation}  
Here $\beta$ and $\mathcal{U}(t,x;M)$ are as above. 
\end{theorem}

\begin{remark}
Note that $u_0 \in L_+^1(\Real^d; (1 + \abs{x}^2)dx )\cap L^\infty(\Real^d)$ implies $u_0 \log u_0 \in L^1(\Real^d)$ by Jensen's inequality. 
\end{remark}

\begin{remark}
The convergence rate in \eqref{ineq:conv_supercrit_IA} is optimal, as it matches that of the corresponding diffusion equation. 
Optimality is not known for \eqref{ineq:conv_IA} or \eqref{ineq:conv_IA2}, however we suspect that these rates are nearly optimal.
Note that the convergence rate obtained in \eqref{ineq:conv_IA2} reduces to \eqref{ineq:conv_supercrit_IA} and \eqref{ineq:conv_IA} when $\gamma = d$. 
Moreover, if $\gamma = d-1$, then the convergence rate goes to zero as $m \nearrow 2 -2/d$. 
\end{remark}

\begin{remark}
In some cases, it may be possible to show that the size conditions of \textit{(i)} of Theorem \ref{thm:Decay} are only required of the $L^{(2-m)d/2}$ norm, which is the norm left invariant by the scaling symmetry of \eqref{def:ADD} when $\K$ is the Newtonian potential.
Indeed, one can prove global existence and uniform boundedness of solutions to \eqref{def:ADD} in these supercritical regimes using iteration techniques without any smallness assumption on the mass \cite{Corrias04,SugiyamaDIE06,SugiyamaADE07}.  
However, in \cite{LionsCC84} it is shown that depending on the singularity of the kernel, if the mass is sufficiently large, there may exist stationary solutions for all $m$, $1 < m < 2$.  
Since these are potentially included in our analysis, in order to state decay results in the generality we have chosen to, a smallness condition on the mass seems to be necessary. 
\end{remark}

\begin{remark}
The results of \cite{BRB10} suggest that if the kernel $\K$ is less singular than the Newtonian potential at the origin, the $L^{(2-m)d/2}$ norm could, in some cases, possibly be replaced by a weaker one. 
\end{remark}

\begin{remark}
In the critical case $m = 2-2/d$, in general $\epsilon_0$ will be strictly less than the critical mass \cite{BlanchetEJDE06,Blanchet09,BRB10}.
As shown in \cite{CalvezCarrillo06,BlanchetEJDE06,Blanchet09,BRB10}, equi-integrability conditions such as \eqref{def:EquiInTheta} allow one to use the dissipated free energy to deduce global control up to the sharp critical mass.
If $\K$ is the Newtonian potential and $d \geq 3$, then the energy dissipation inequality is sufficient to imply \eqref{def:EquiInTheta} and therefore solutions satisfy \eqref{ineq:LinftyDecay} all the way up to the critical mass \cite{Blanchet09}.  
Unfortunately, the energy dissipation inequality does not always seem to imply \eqref{def:EquiInTheta} in general. 
\end{remark}

\begin{remark}
We consider only the case of power-law diffusion, however, the estimate \eqref{ineq:LinftyDecay} holds for \eqref{def:ADD_general} provided $A^\prime(z) \geq cz^{m-1}$ for some $c > 0$. 
Therefore, it is likely possible to apply the methods of \cite{BilerDolbeaultEsteban02,CarrilloDiFranToscani06} to this more general case under 
some structural assumptions on $A$.
\end{remark}

\subsection{Outline of Proof} \label{sec:outline}  
The proof of Theorems \ref{thm:IA} and \ref{thm:IA2} involves several steps. 
As mentioned above, we use the entropy-entropy dissipation methods of \cite{CarrilloToscani98,CarrilloToscani00,CarrilloEntDiss01} and in particular, the time-dependent rescaling used in \cite{CarrilloToscani00}. 
All of the computations will be formal, they can be made rigorous for weak solutions either with a suitable parabolic regularization and passing to the limit, as in for instance \cite{BertozziSlepcev10,BRB10,CarrilloEntDiss01,BRB_2D_10}, or presumably also lifting to strictly positive solutions, as is common in the study of the porous media equation \cite{VazquezPME}.  

Following \cite{CarrilloToscani00}, we define $\theta(\tau,\eta)$ such that
\begin{equation}
e^{-d\tau}\theta(\tau,\eta) = u(t,x), \label{def:theta}
\end{equation}
with coordinates $e^\tau\eta = x$ and $\beta e^{\beta^{-1}\tau} - \beta = t$, where $\beta$ is given by \eqref{def:beta}. In what follows we denote $\alpha := d\beta$. 
In these coordinates, if $u(t,x)$ solves \eqref{def:ADD} then $\theta(\tau,\eta)$ solves, 
\begin{equation}
\partial_\tau \theta = \grad \cdot (\eta \theta) + \Delta \theta^m - e^{(1 - \alpha - \beta)\beta^{-1}\tau} \grad \cdot(\theta (\conv{e^{d\tau}\grad\K(e^{\tau} \cdot)}{\theta})). \label{eq:asymptotic_pde}
\end{equation}
Moreover, $\mathcal{U}(t,x;M)$ is stationary in these coordinates, and will be denoted by $\theta_M(\eta)$. 
That is (see \cite{CarrilloToscani00}), 
\begin{equation}
\mathcal{U}(t,x;M) = \left(1 + \frac{t}{\beta}\right)^{-d\beta}\theta_M\left( \left(1 + \frac{t}{\beta}\right)^{-\beta}x\right) = e^{-d\tau}\theta_M(\eta). \label{eq:BarenblattToGS}
\end{equation}
In fact, $\theta_M(\eta)$ is the unique non-negative solution with mass $M$ to the (degenerate, if $m > 1$) elliptic equation
\begin{equation}
0 = \grad \cdot (\eta \theta) + \Delta \theta^m. \label{eq:ground_state}
\end{equation}
In what follows we will refer to $\theta_M$ as the \emph{ground state Barenblatt solution}. 
Clearly, ground state solutions are stationary solutions of the homogeneous Fokker-Plank equation 
\begin{equation}
\partial_\tau \theta = \grad \cdot (\eta \theta) + \Delta \theta^m. \label{def:HFkP} 
\end{equation} 
Therefore, the asymptotic convergence to self-similar profiles of solutions to \eqref{def:PME} is equivalent to the convergence to the stationary profiles of \eqref{def:HFkP}. 
This was the fundamental observation made in \cite{CarrilloToscani00} and is the purpose of the rescaling \eqref{def:theta}.

A primary step to proving Theorems \ref{thm:IA} and \ref{thm:IA2} is establishing
that $\theta(\tau,\eta) \in L^\infty_{\tau,\eta}(\Real^+ \times \Real^d)$. 
By the change of variables, this is precisely the decay estimate \eqref{ineq:LinftyDecay} stated in Theorem \ref{thm:Decay}. 
This estimate is what allows us to treat the inhomogeneous non-local term in \eqref{eq:asymptotic_pde} 
as a vanishing perturbation of \eqref{def:HFkP}. 
The decay estimate $\norm{u(t)}_\infty \lesssim t^{-d\beta}$, or equivalently, $\norm{\theta(\tau)}_\infty \lesssim 1$, is easily obtained for \eqref{def:PME} in the linear case and the classical Aronson-B\'enilan estimate proves it in the case $m > 1$ \cite{VazquezPME}. 
Clearly, no such analogues are available for \eqref{eq:asymptotic_pde}, and they will instead be provided by Theorem \ref{thm:Decay}.  
To prove Theorem \ref{thm:Decay}, we adapt the Alikakos iterations of \cite{Kowalczyk05,CalvezCarrillo06,Blanchet09,BRB10} to \eqref{eq:asymptotic_pde} to prove a uniform bound in the rescaled variables. 
Dealing with time-dependent rescalings in \eqref{eq:asymptotic_pde} introduces several complications. 
Obtaining $L^p$ estimates for the critical case $m = 2-2/d$ is relatively straightforward due to the inherent 
scale-invariance of the relevant inequalities.
In the supercritical case $m < 2-2/d$, the effect of the time-dependent rescaling in \eqref{eq:asymptotic_pde} 
is crucial for closing a key bootstrap/continuity argument necessary to control the solution uniformly in time.
It is at this step that the method below diverges significantly from existing methods and is a key step to obtaining Theorem \ref{thm:Decay}.   
The additional issue when $\grad \K \not\in L^1$ is in obtaining $L^\infty$ estimates, which requires measuring the rate at which $e^{d\tau}\grad \K(e^{\tau}\eta)$ blows up in $L^1_{loc}$ as $\tau \rightarrow \infty$. 
This also arises later when we estimate how much the nonlocal term affects the entropy dissipation and is the source of the degraded convergence rates found in Theorem \ref{thm:IA2}.  
 
Once we have established $\theta(\tau,\eta) \in L^\infty_{\tau,\eta}(\Real^+ \times \Real^d)$, we prove that solutions to \eqref{eq:asymptotic_pde} converge to $\theta_M$ 
and estimate the convergence rate in $L^1$. 
In fact, these are done together, as the quantitative estimate is direct and removes the need for compactness arguments. 
The primary step of the entropy dissipation method is an estimate of the decay of the entropy associated to \eqref{def:HFkP}. 
In the case $m = 1$, the entropy is given by, 
\begin{equation}
H(\theta) = \int \theta \log \theta d\eta + \frac{1}{2}\int\abs{\eta}^2\theta d\eta, \label{def:H_linear}
\end{equation}
and the entropy production functional by 
\begin{equation}
I(\theta) = \int \theta \abs{\grad \log \theta +  \eta}^2 d\eta. \label{def:I_linear} 
\end{equation}
In the nonlinear case $m > 1$, the corresponding quantities are, 
\begin{equation}
H(\theta) = \frac{1}{m-1}\int \theta^m d\eta + \frac{1}{2}\int \abs{\eta}^2 \theta d\eta, \label{def:H}
\end{equation}
and the entropy production functional, 
\begin{equation}
I(\theta) = \int u\abs{\frac{m}{m-1}\grad u^{m-1} + \eta}^2 d\eta. \label{def:I}
\end{equation}
In the nonlinear case, these entropies were originally introduced for studying \eqref{def:HFkP} in \cite{Newman84,Ralston84}.
Both \eqref{def:H_linear} and \eqref{def:H} are displacement convex \cite{McCann97} and in fact,
\eqref{def:HFkP} is a gradient flow for \eqref{def:H} or \eqref{def:H_linear} in the Euclidean Wasserstein distance \cite{Otto01,AmbrosioGigliSavare,CarrilloMcCannVillani03,CarrilloMcCannVillani06}, and 
if $f(\tau,\eta)$ solves \eqref{def:HFkP}, then 
\begin{equation*}
\frac{d}{d\tau}H(f(\tau)) = -I(f(\tau)). 
\end{equation*}
For a given mass $M$, \eqref{def:H} has a unique non-negative minimizer which is the ground state $\theta_M$. 
That is, if we define the relative entropy 
\begin{equation}
H(\theta|\theta_M) = H(\theta) - H(\theta_M),  \label{def:relativeH}
\end{equation}
then $H(\theta|\theta_M) \geq 0$ with equality if and only if $\theta = \theta_M$ \cite{CarrilloToscani00,DelPinoDolbeault02}. 
In order to estimate a convergence rate, it is therefore sensible to measure how quickly $H(\theta|\theta_M) \rightarrow 0$.
Following the methods of \cite{CarrilloToscani00,CarrilloEntDiss01,CarrilloMcCannVillani03,CarrilloMcCannVillani06}, this is made possible by the following two theorems. 
The first relates 
the entropy production functional \eqref{def:I} to the relative entropy \eqref{def:relativeH}.
This represents a generalization of the Gross logarithmic inequality \cite{Gross75} (see also \cite{DelPinoDolbeault02}).

\begin{theorem}[Generalized Gross Logarithmic Sobolev Inequality \cite{CarrilloToscani00,CarrilloEntDiss01,DelPinoDolbeault02,Gross75}] \label{thm:rel_entropy}
Let $f\in L_+^1(\Real^d)$ with $\norm{f}_1 = M$ and let $\theta_M$ be the ground state Barenblatt solution with mass $M$. Then, 
\begin{equation}
H(f|\theta_M)  \leq \frac{1}{2} I(f). \label{ineq:relative_entropy}
\end{equation}   
\end{theorem}  

For the Fokker-Plank equation \eqref{def:HFkP}, Theorem \ref{thm:rel_entropy} implies  $H(\theta(\tau)|\theta_M) \lesssim e^{-2\tau}$.
The (generalized) Csiszar-Kullback inequality \cite{Csiszar67,Kullback67} relates the relative entropy to the $L^1$ norm.

\begin{theorem}[Csiszar-Kullback Inequality \cite{CarrilloEntDiss01}] \label{thm:CK}
Let $f\in L_+^1(\Real^d)$ with $\norm{f}_1 = M$ and let $\theta_M$ be the ground state Barenblatt solution with mass $M$. 
Then,  
\begin{equation}
\norm{f - \theta_M}_1 \lesssim H(f|\theta_M)^{\min\left(\frac{1}{2},\frac{1}{m}\right)}. 
\end{equation}   
\end{theorem} 
Note that since we are interested in $1 \leq m \leq 2-2/d$, we will only apply the inequality with exponent $1/2$. 

To prove Theorems \ref{thm:IA} and \ref{thm:IA2}, the purpose of proving $\theta(\tau,\eta) \in L^\infty_{\tau,\eta}(\Real^+ \times \Real^d)$ is to control 
the growth of $\norm{e^{d\tau}\grad\K(e^{\tau}\cdot)\ast \theta}_\infty$, which depends on the long-range effects of the kernel. Ultimately, this provides a bound essentially of the form, 
\begin{equation*}
\frac{d}{d\tau}H(\theta(\tau)) \leq -I(\theta(\tau)) + C(M,\norm{\theta}_{L_{\tau,\eta}^\infty(\Real^+ \times \Real^d)})e^{-\gamma \tau}, 
\end{equation*}
for some $\gamma > 0$ (in reality, it is not quite as clean). Theorem \eqref{thm:rel_entropy} then implies, 
\begin{equation*}
\frac{d}{d\tau}H(\theta(\tau)|\theta_M) \leq -2H(\theta(\tau)|\theta_M) + C(M,\norm{\theta}_{L_{\tau,\eta}^\infty(\Real^+ \times \Real^d)})e^{-\gamma \tau}. 
\end{equation*}
Integrating this and applying Theorem \ref{thm:CK} implies,
\begin{equation*}
\norm{\theta - \theta_M}_1 \lesssim e^{-\frac{\tau}{2}\min\left(2,\gamma\right)}, 
\end{equation*} 
which after rescaling and interpolation against the decay estimate \eqref{ineq:LinftyDecay}, will prove Theorems \ref{thm:IA} and \ref{thm:IA2}. 

\section{Preliminary Decay Estimates} 
Let $\overline{q} = (2-m)d/2$ and let $\eta, \tau$ and $\theta(\tau,\eta)$ be as defined in \S\ref{sec:outline}. 
As detailed above, we establish that $\theta(\tau,\eta) \in L^\infty_{\tau,\eta}(\Real^+\times\Real^d)$ using Alikakos iteration \cite{Alikakos} (see also \cite{JagerLuckhaus92,Kowalczyk05,BRB10,Blanchet09,SugiyamaADE07,SugiyamaDIE07,SugiyamaDIE06}). 
The first step is to prove the following lemma which allows control over $L^p$ norms with $p < \infty$. In what follows we denote $\theta_0(\eta) := \theta(\eta,0) = u(x,0)$. 
\begin{lemma}[Control for $L^p$, $p < \infty$ for small data] \label{lem:finite_p_bounded}
For all $\overline{q} \leq p < \infty$, there exists $C_{\overline{q}} = C_{\overline{q}}(p,M)$ and $C_M = C_M(p,\norm{\theta_0}_{\overline{q}})$ such that if $\norm{\theta_0}_{\overline{q}} < C_{\overline{q}}$ and $M < C_M$, then 
$\norm{\theta(\tau)}_{p} \in L_\tau^\infty(\Real^+)$. 
\end{lemma}
\begin{proof}
Define 
\begin{equation*}
\mathcal{I} = \int \theta^{m-1}\abs{\grad \theta^{p/2}}^2 dx. 
\end{equation*}
We estimate the time evolution of $\norm{\theta}_p$ using integration by parts, H\"older's inequality and Lemma \ref{lem:CZ_rescale} in the appendix, 
\begin{align}
\frac{d}{d\tau} \norm{\theta}_{p}^{p} & = -\frac{4mp}{(p+1)^2}\mathcal{I} + (p-1)e^{(1-\alpha-\beta)\beta^{-1}\tau}\int\theta^p\grad \cdot (e^{d\tau}\grad \K(e^{\tau} \cdot) \ast \theta) d\eta + d(p-1)\norm{\theta}_p^p \nonumber \\
& \leq -C(p)\mathcal{I} + C(p)e^{(1-\alpha-\beta)\beta^{-1}\tau}\norm{\theta}_{p+1}^p\norm{\grad(e^{d\tau}\grad\K(e^{\tau} \cdot)\ast \theta)}_{p+1} + C(p)\norm{\theta}_p^p \nonumber \\
& \leq  -C(p)\mathcal{I} + C(p)e^{(1-\alpha)\beta^{-1}\tau}\norm{\theta}_{p+1}^{p+1} + C(p)\norm{\theta}_p^p. \label{ineq:Lpevo}
\end{align}
We bound the second term using the using the homogeneous Gagliardo-Nirenberg-Sobolev inequality (Lemma \ref{lem:GNS} in appendix),
\begin{equation}
\norm{\theta}_{p+1}^{p+1} \lesssim \norm{\theta}_{\overline{q}}^{\alpha_2(p+1)}\mathcal{I}^{\alpha_1(p+1)/2}, \label{ineq:GNS1}
\end{equation}
where $\alpha_2 = 1 - \alpha_1(p+m-1)/2$ and 
\begin{equation*}
\alpha_1 = \frac{2d(\overline{q}-p-1)}{(p+1)\left(\overline{q}(d-2) - d(p+m-1)\right)}. 
\end{equation*}
By the definition of $\overline{q}$ we have that, 
\begin{equation*}
\frac{\alpha_1(p+1)}{2} = \frac{d(\overline{q}-p-1)}{\overline{q}(d-2) - d(p+m-1)} = 1. 
\end{equation*}
We also estimate the second term in \eqref{ineq:Lpevo} using Lemma \ref{lem:GNS}, 
\begin{equation}
\norm{\theta}_{p}^{p} \lesssim M^{\beta_2 p}\mathcal{I}^{\beta_1p/2}, \label{ineq:gns_below}
\end{equation}
where $\beta_2 = 1 - \beta_1p/2$ and, 
\begin{equation*}
\frac{\beta_1 p}{2} = \frac{d(p-1)}{2-d+d(p+m-1)} < 1,
\end{equation*}
by $1 - 2/d < m$. Then applying weighted Young's inequality we have from \eqref{ineq:GNS1}, \eqref{ineq:gns_below} and \eqref{ineq:Lpevo}, 
\begin{equation}
\frac{d}{d\tau} \norm{\theta}_{p}^{p} \leq \left(C_1(p)e^{(1-\alpha)\beta^{-1}\tau}\norm{\theta}_{\overline{q}}^{\alpha_2(p+1)} - C_2(p) \right)\mathcal{I} + C_3(p)M^{\gamma(p)}, \label{ineq:theta_p}
\end{equation}
for $\gamma(p) = 2\beta_2p/(2-\beta_1p) > 0$.
If $m = 2- 2/d$, then $\overline{q} = 1$ and $1 - \alpha = 0$, therefore by conservation of mass it is possible to choose $M$ sufficiently small such that the first term in \eqref{ineq:theta_p} is less than $-\delta \mathcal{I}$ for some $\delta > 0$.
If $m < 2 -2/d$, then $\overline{q} > 1$ and $\norm{\theta}_{\overline{q}}$ is no longer conserved. Here we must take advantage of $1 - \alpha < 0$. Note that \eqref{ineq:theta_p} holds for $p = \overline{q}$; therefore since $1 - \alpha < 0$, a continuity argument establishes that for $\norm{\theta_0}_{\overline{q}}$ and $M$ sufficiently small, 
\begin{equation*}
\norm{\theta(\tau)}_{\overline{q}}^{\overline{q}} \leq \norm{\theta_0}_{\overline{q}}^{\overline{q}} + C_3(\overline{q})M^{\gamma(\overline{q})} \tau. 
\end{equation*}
Indeed, for $\norm{\theta_0}_{\overline{q}}$ small, this holds for at least some time, and for $M$ sufficiently small, this linear growth is such that the first term in \eqref{ineq:theta_p} remains non-positive forever. 
Then by \eqref{ineq:theta_p} for $p > \overline{q}$, if $M$ and $\norm{\theta_0}_{\overline{q}}$ additionally satisfy
\begin{equation*}
C_1(p)e^{(1-\alpha)\beta^{-1}\tau}(C_3(\overline{q})M^{\gamma(\overline{q})}\tau + \norm{\theta_0}_{\overline{q}}^{\overline{q}})^{\alpha_2(p+1)/\overline{q}} - C_2(p) < -\delta, 
\end{equation*}
for all $\tau > 0$, then the first term is less than $-\delta I$. By $1-\alpha < 0$ we may always choose $M$ and $\norm{\theta_0}_{\overline{q}}$ such that this is possible.
Therefore, whether $\overline{q} > 1$ or $\overline{q} = 1$, for small initial data in the suitable sense, we have
\begin{equation*}
\frac{d}{d\tau} \norm{\theta}_{p}^{p} \leq -\delta \mathcal{I} + C(M,p). 
\end{equation*} 
Using \eqref{ineq:gns_below} and Young's inequality for products, we have a lower bound on $\mathcal{I}$,
\begin{equation*}
\norm{\theta}_{p}^{p} - C(M) \leq \mathcal{I}. 
\end{equation*}
This proves,
\begin{equation*}
\frac{d}{d\tau} \norm{\theta}_{p}^{p} \leq -\delta\norm{\theta}_p^p + C(M,p), 
\end{equation*}
which immediately concludes the lemma with $\norm{\theta}_p^p \leq \max(\norm{\theta_0}_p^p, C(M,p)\delta^{-1})$.  
\end{proof}

We now turn to proving that \eqref{def:EquiInTheta} implies something analogous to Lemma \ref{lem:finite_p_bounded}. 
Let $u(t)$ be as in \textit{(ii)} of Theorem \ref{thm:Decay}. One can verify that \eqref{def:EquiInTheta} is equivalent to 
\begin{equation}
\lim_{k \rightarrow \infty} \sup_{\tau \in [0, \infty)} \norm{\left(\theta(\tau) - k\right)_+}_1 = 0, \label{def:theta_equi}
\end{equation}
which is precisely the condition of uniform equi-integrability which plays a key role in \cite{CalvezCarrillo06,Blanchet09,BRB10}. 
We may refine Lemma \ref{lem:finite_p_bounded} in the following fashion, adapting the techniques in \cite{CalvezCarrillo06,Blanchet09,BRB10} to this setting.   
\begin{lemma}[Control for $L^p$, $p < \infty$ for equi-integrable solutions] \label{lem:finite_p_bounded_unifint}
If $\theta(\tau)$ satisfies \eqref{def:theta_equi} then we have $\norm{\theta(\tau)}_{p} \in L_\tau^\infty(\Real^+)$ for all $p < \infty$. 
\end{lemma}
\begin{proof}
We proceed similar to the proof of Lemma \ref{lem:finite_p_bounded}, but now slightly refined to take advantage of \eqref{def:theta_equi}. Since similar arguments have appeared in several locations (for example \cite{CalvezCarrillo06,Blanchet09,BRB10}) we sketch a proof and highlight mainly the differences that appear due to the rescaling in \eqref{eq:asymptotic_pde}.  
Define $\theta_k(\tau,\eta) := (\theta(\tau,\eta) - k)_+$ and 
\begin{equation*}
\mathcal{I} = \int \theta_k^{m-1}\abs{\grad \theta_k^{p/2}}^2 dx. 
\end{equation*}
The $L^p$ norms of $\theta$ and $\theta_k$ are related through the following inequality for $1 \leq p < \infty$, 
\begin{equation}
\norm{\theta}_p^p \lesssim_p \norm{\theta_k}_p^p + k^{p-1}\norm{\theta}_1. \label{ineq:slicing}
\end{equation}
It is important to note that the implicit constant in \eqref{ineq:slicing} does not depend on $k$.  
Estimating the time evolution of $\theta_k$ as in Lemma \ref{lem:finite_p_bounded}, using Lemma \ref{lem:CZ_rescale} and \eqref{ineq:slicing} implies, 
\begin{align*}
\frac{d}{d\tau}\norm{\theta_k}_p^p & = -C(p)\mathcal{I} - \int \left((p-1)\theta_k^{p} + kp\theta^{p-1}_k \right) \grad \cdot \left(e^{d\tau}\grad \K(e^{\tau}\cdot) \ast \theta_k\right)d\eta \\ 
& \leq -C(p)\mathcal{I} + C(p)\norm{\theta_k}_{p+1}^{p+1} + C(p,k)\norm{\theta_k}_p^p + C(k,p,M). 
\end{align*}
Using the Gagliardo-Nirenberg-Sobolev inequality (Lemma \ref{lem:GNS}) implies, 
\begin{equation*}
\frac{d}{d\tau}\norm{\theta_k}_p^p \leq -\frac{C(p)}{\norm{\theta_k}^{\alpha_2}_1}\norm{\theta_k}_{p+1}^{p+1} + C(p)\norm{\theta_k}_{p+1}^{p+1} + C(p,k)\norm{\theta_k}_p^p + C(k,p,M),  
\end{equation*}
where $\alpha_2 = 1 - \alpha_1(p+m-1)/2 > 0$ and 
\begin{equation*}
\alpha_1 = \frac{2d(1 - 1/(p+1))}{2 - d + dp + d(m-1)}. 
\end{equation*}
Note $\norm{\theta_k}_p \leq M^{1/p^2}\norm{\theta_k}_{p+1}^{(p^2 -1)/p^2}$, which by weighted Young's inequality implies, 
\begin{equation*}
\frac{d}{d\tau}\norm{\theta_k}_p^p \leq -\frac{C(p)}{\norm{\theta_k}^{\alpha_2}_1}\norm{\theta_k}_{p+1}^{p+1} + C(p)\norm{\theta_k}_{p+1}^{p+1} + C(k,p,M).  
\end{equation*}
Using \eqref{def:theta_equi} we may make the leading order terms as negative as we want and interpolating $L^p$ against $L^1$ and $L^{p+1}$ again implies there is a $\delta > 0$ such that if $k$ is sufficiently large we have, 
\begin{equation*}
\frac{d}{d\tau}\norm{\theta_k}_p^p \leq -\delta\norm{\theta_k}_p^p + C(k,p,M). 
\end{equation*}
By \eqref{ineq:slicing} and conservation of mass, this concludes the proof of Lemma \ref{lem:finite_p_bounded_unifint}.
\end{proof} 

\section{Finite Length-Scales}
We begin by proving Theorem \ref{thm:Decay} for the case $\grad \K \in L^1$. 
Alikakos iteration \cite{Alikakos} is a standard method for using a result such as Lemma \ref{lem:finite_p_bounded} to imply a result of the following form.  
\begin{lemma}[Control of $L^\infty$ for small data] \label{lem:rescaled_inftybdd}
Let $\grad \K \in L^1$. Then there exists $C_{\overline{q}} = C_{\overline{q}}(M)$ and $C_M = C_M(\norm{\theta_0}_{\overline{q}})$ such that if $\norm{\theta_0}_{\overline{q}} < C_{\overline{q}}$ and $M < C_M$, then 
$\norm{\theta(\tau)}_{\infty} \in L_\tau^\infty(\Real^+)$. 
\end{lemma}
\begin{proof}
Standard iteration implies $\norm{\theta(\tau)}_\infty \in L_\tau^\infty(\Real^+)$, provided 
\begin{equation*}
\vec{v} := e^{(1 - \alpha - \beta)\beta^{-1}\tau}\conv{e^{d\tau}\grad \K(e^{\tau} \cdot)}{\theta} \in L^\infty_{\tau,\eta}(\Real^+\times \Real^d). 
\end{equation*}
See \cite{JagerLuckhaus92,CalvezCarrillo06,BRB10,Kowalczyk05,SugiyamaDIE06,SugiyamaADE07}. For instance, an iteration lemma due to Kowalczyk \cite{Kowalczyk05} may be extended easily 
to the case $\Real^d$, $d \geq 2$ and to include the $\grad \cdot (\eta \theta)$ term in \eqref{eq:asymptotic_pde} \cite{CalvezCarrillo06}.

Fix $p > d$. Then by Lemma \ref{lem:finite_p_bounded}, for sufficiently small $M$ and $\norm{\theta_0}_{\overline{q}}$, $\norm{\theta(\tau)}_p \in L_\tau^\infty(\Real^+)$. 
Therefore by Lemma \ref{lem:CZ_rescale} in the appendix,  
\begin{equation*}
\norm{\grad \vec{v}}_p = \norm{e^{(1-\alpha-\beta)\beta^{-1}\tau}\grad\left(e^{d\tau}\grad\K(e^{\tau} \cdot)\ast\theta\right)}_p \lesssim e^{(1-\alpha)\beta^{-1}\tau}\norm{\theta}_p \lesssim e^{(1-\alpha)\beta^{-1}\tau}.  
\end{equation*}
Moreover, by $\grad \K \in L^1(\Real^d)$, 
\begin{align*}
\norm{\vec{v}}_p & \leq e^{(1 - \alpha - \beta)\beta^{-1}\tau}\norm{\theta}_p \lesssim e^{(1 - \alpha - \beta)\beta^{-1}\tau}. 
\end{align*}
Since $1 - \alpha \leq 0$, Morrey's inequality implies $\vec{v} \in L^\infty_{\tau,\eta}(\Real^+\times\Real^d)$ and the lemma follows. 
\end{proof}

By Lemma \ref{lem:rescaled_inftybdd} and the definition of $\tau$, 
\begin{equation*}
\norm{u(t)}_{L_x^\infty(\Real^d)} = e^{-d\tau}\norm{\theta}_{L_\eta^\infty(\Real^d)} \lesssim (1+t)^{-d\beta},
\end{equation*}
establishing \eqref{ineq:LinftyDecay}. 
A similar argument using Lemma \ref{lem:finite_p_bounded_unifint} in place of Lemma \ref{lem:finite_p_bounded} implies
\begin{lemma}
Theorem \ref{thm:Decay} holds if $\grad \K \in L^1$.
\end{lemma}

Now we turn to Theorem \ref{thm:IA}.
\begin{proof}(Theorem \ref{thm:IA}: \textbf{Intermediate Asymptotics I})
Now that the requisite decay estimate has been established, we proceed by estimating the decay of the relative entropy \eqref{def:relativeH}. 
By Young's inequality, $\grad \K \in L^1(\Real^d)$ and \ref{ineq:LinftyDecay}, 
\begin{equation}
\norm{e^{d\tau}\grad\K(e^{\tau}\cdot)\ast\theta}_\infty \leq \norm{\grad \K}_1\norm{\theta}_\infty \lesssim 1.  \label{ineq:velocity_bounded} 
\end{equation}

We first settle the case $m > 1$. 
By a standard computation, \eqref{ineq:velocity_bounded} and Cauchy-Schwarz, for all $\delta > 0$, 
\begin{align*} 
\frac{d}{d\tau}H(\theta(\tau)|\theta_M) & = -I(\theta) +  e^{(1-\alpha-\beta)\beta^{-1}\tau}\int \grad\left( \frac{1}{m-1}\theta^{m} + \frac{1}{2}\abs{\eta}^2\right)\cdot \theta e^{d\tau}\grad\K(e^{\tau}\cdot)\ast\theta d\eta \\
& \leq -I(\theta) + e^{(1-\alpha-\beta)\beta^{-1}\tau}I(\theta)^{1/2} \left(\int \theta \abs{e^{d\tau}\grad\K(e^{\tau}\cdot)\ast\theta}^2 d\eta \right)^{1/2} \\
& \leq -(1-e^{-2\delta\tau})I(\theta) + Ce^{(2-2\alpha-2\beta)\beta^{-1}\tau + 2\delta\tau}. 
\end{align*} 
Let $\gamma(\delta) := (2\alpha+2\beta - 2)\beta^{-1} - 2\delta > 0$. By the generalized Gross Logarithmic Sobolev inequality, Theorem \ref{thm:rel_entropy}, we therefore have, 
\begin{equation}
\frac{d}{d\tau}H(\theta(\tau)|\theta_M) \leq -2(1 - e^{-2\delta\tau})H(\theta|\theta_M) + Ce^{-\gamma \tau}. \label{ineq:Htheta_RoC}
\end{equation}
Solving the differential inequality \eqref{ineq:Htheta_RoC} implies, 
\begin{equation*}
H(\theta(\tau)|\theta_M) \lesssim e^{-\tau\min\left(2,\gamma(\delta)\right)}. 
\end{equation*}
Now by the generalized Csiszar-Kullback inequality, Theorem \ref{thm:CK},
\begin{equation*}
\norm{\theta(\tau) - \theta_M}_1 \lesssim e^{-\frac{\tau}{2}\min\left(2,\gamma(\delta)\right)}.
\end{equation*}
Re-writing in terms of $x$ and $t$ and using \eqref{eq:BarenblattToGS}, 
\begin{equation*}
\norm{u(t) - \mathcal{U}(t;M)}_1 \lesssim (1+t)^{-\frac{\beta}{2}\tau\min\left(2,\gamma(\delta)\right)}.
\end{equation*}
If $m < 2-2/d$, it can be verified that $\delta>0$ may always be chosen small enough such that $2 < \gamma(\delta)$. 
If instead $m = 2-2/d$, then $2d + 2 -2\beta^{-1} = 2$. This establishes \eqref{ineq:conv_IA} in the case $p = 1$. 
Interpolation against \eqref{ineq:LinftyDecay} completes the proof. 

We now settle the case $m = 1$. 
The time evolution of the relative entropy is similar to above. By \eqref{ineq:velocity_bounded} and Cauchy-Schwarz, for all $\delta > 0$, 
\begin{align*} 
\frac{d}{d\tau}H(\theta(\tau)|\theta_M) & = -I(\theta) +  e^{(1-\alpha-\beta)\beta^{-1}\tau}\int \grad\left( \log \theta + \frac{1}{2}\abs{\eta}^2\right)\cdot \theta e^{d\tau}\grad\K(e^{\tau}\cdot)\ast\theta d\eta \\
& \leq -I(\theta) + e^{(1-\alpha-\beta)\beta^{-1}\tau}I(\theta)^{1/2} \left(\int \theta \abs{e^{d\tau}\grad\K(e^{\tau}\cdot)\ast\theta}^2 d\eta \right)^{1/2} \\
& \leq (1-e^{-\delta\tau})I(\theta) + Ce^{(2-2\alpha-2\beta)\beta^{-1}\tau + \delta\tau}. 
\end{align*} 
The rest of the proof follows similarly to the case $m > 1$ using Theorems \ref{thm:rel_entropy} and \ref{thm:CK}. 
This concludes the proof of Theorem \ref{thm:IA}. 
\end{proof}

\begin{remark}
A generalization of Talagrand's inequality \cite{CarrilloMcCannVillani03} shows that $\theta \rightarrow \theta_M$ also in the Euclidean Wasserstein distance.
\end{remark}

\section{Infinite Length-Scales}

We now turn to the proofs of Theorem \ref{thm:Decay} and Theorem \ref{thm:IA2} in the case $\grad \K \not\in L^1$. 
In order to properly extend the work of the previous section, we must estimate the quantities $\norm{e^{d\tau}\grad \K(e^{\tau}\cdot)\ast\theta}_p$ appearing in \eqref{ineq:velocity_bounded} and the proof of Lemma \ref{lem:rescaled_inftybdd}. 
However, $\grad \K \not\in L^1(\Real^d)$ and Young's inequality is not sufficient; in fact we will not bound $\norm{e^{d\tau}\grad \K(e^{\tau}\cdot)\ast\theta}_p$ uniformly in time but instead bound the rate at which it grows.
We separately estimate the growth of the quantities 
$\norm{\lambda^d\grad\K(\lambda \cdot)\mathbf{1}_{B_1(0)}}_{1}$ and 
$\norm{\lambda^d\grad\K(\lambda \cdot)\mathbf{1}_{\Real^d\setminus B_1(0)}}_{p}$ as $\lambda \rightarrow \infty$. 
Using $\abs{\grad \K(x)} \lesssim \abs{x}^{-\gamma}$ for sufficiently large $\abs{x}$, if $\gamma < d$, then for large $\lambda$,    
\begin{align}
\int \lambda^d\abs{\grad \K(\lambda y)}\mathbf{1}_{B_1(0)}(\abs{y}) dy & = \int_{\abs{y} \leq \lambda} \abs{\grad \K(y)} dy \nonumber \\
& = \int_{S^{d-1}}\int_0^{\lambda} \abs{\grad \K(\rho\omega)}r\rho^{d-1} d\rho d\omega \nonumber \\ 
& \lesssim 1 + \lambda^{d-\gamma}. \label{ineq:KL1loc}
\end{align}
Similarly, if $\gamma = d$, then for large $\lambda$, 
\begin{equation}
\int \lambda^d\abs{\grad \K(\lambda y)}\mathbf{1}_{B_1(0)}(\abs{y}) dy \lesssim 1 + \log \lambda. \label{ineq:KL1loc_log}
\end{equation}
If $d/(d-1) < q < \infty$, since $\gamma \geq d-1$, for $\lambda$ sufficiently large we have, 
\begin{align}
\int \lambda^{qd}\abs{\grad \K(\lambda y)}^q\mathbf{1}_{\Real^d \setminus B_1(0)}(\abs{y}) dy & = \int_{\abs{y} \geq \lambda} \lambda^{qd - d} \abs{\grad \K(y)}^q dy \nonumber \\
& = \lambda^{qd - d} \int_{S^{d-1}}\int_\lambda^{\infty} \abs{\grad \K(\rho\omega)}^q\rho^{d-1} d\rho d\omega \nonumber \\
& \lesssim \lambda^{q(d-\gamma)}. \label{ineq:K_Lq_decay}
\end{align}
Similarly,
\begin{equation}
\sup_{\abs{x} \geq 1}\abs{\lambda^d\grad \K(\lambda x)} \lesssim 1 + \lambda^{d-\gamma}. \label{ineq:K_Linfty_decay}
\end{equation}

We may now complete the general proof of Theorem \ref{thm:Decay}.
\begin{proof}(\textbf{Theorem} \ref{thm:Decay})
We first complete the proof of \textit{(i)}.
Lemma \ref{lem:rescaled_inftybdd} extends to the case $\grad \K \not\in L^1$ provided we can bound $\vec{v}  := e^{(1-\alpha-\beta)\beta^{-1}\tau} e^{d\tau}\grad \K(e^{\tau} \cdot) \ast \theta$ in $L^\infty_{\eta}(\Real^d)$ uniformly in time.
Indeed, fix $p > d$. Then for $M$ and $\norm{\theta_0}_{\overline{q}}$ sufficiently small, we have by Lemma \ref{lem:finite_p_bounded},
$\norm{\theta(\tau)}_p \in L^\infty_\tau(\Real^+)$. By Lemma \ref{lem:CZ_rescale}, 
\begin{equation*}
\norm{\grad \vec{v}}_p \lesssim e^{(1-\alpha)\beta^{-1}\tau}\norm{\theta}_p \lesssim e^{(1-\alpha)\beta^{-1}\tau}.
\end{equation*} 
Let $q$ be such that $d/(d-1) < q \leq p$, which implies $\norm{\theta(\tau)}_q \lesssim 1$. 
If $\gamma < d$ then by Young's inequality, 
\begin{align*}
\norm{\vec{v}}_{q} & \leq e^{(1-\alpha-\beta)\beta^{-1}\tau}\left( \norm{ e^{d\tau}\grad \K(e^{\tau}\cdot)\mathbf{1}_{B_1(0)} \ast \theta}_q + \norm{ e^{d\tau}\grad \K(e^{\tau}\cdot)\mathbf{1}_{\Real^d \setminus B_1(0)} \ast \theta}_q \right) \\
& \leq  e^{(1-\alpha-\beta)\beta^{-1}\tau}\left( \norm{e^{d\tau}\grad \K(e^{\tau}\cdot)\mathbf{1}_{B_1(0)}}_1\norm{\theta}_q + \norm{e^{d\tau}\grad \K(e^{\tau}\cdot)\mathbf{1}_{\Real^d \setminus B_1(0)}}_qM\right). 
\end{align*}
Since $\norm{\theta(\tau)}_q \lesssim 1$, by \eqref{ineq:KL1loc} and \eqref{ineq:K_Lq_decay} we have, 
\begin{align*}
\norm{\vec{v}}_q & \lesssim e^{(1-\alpha-\beta)\beta^{-1}\tau}\left(1 + e^{(d-\gamma)\tau} \right) \\
& \lesssim e^{(1-\alpha - \beta)\beta^{-1}\tau} + e^{(1-\beta - \gamma\beta)\beta^{-1}\tau}.  
\end{align*}
Since $1 - \beta - \gamma\beta \leq 0$ and $1 - \alpha \leq 0$, by Morrey's inequality we may conclude $\vec{v} \in L^\infty_{\tau,\eta}(\Real^+ \times \Real^d)$. 
Similarly if $\gamma = d$, then by the same reasoning as above, \eqref{ineq:KL1loc_log} and \eqref{ineq:K_Lq_decay} imply,
\begin{align*}
\norm{\vec{v}}_q & \lesssim e^{(1-\alpha - \beta)\beta^{-1}\tau}\left(1 + \tau + e^{(d-\gamma)\tau} \right) \\
& \lesssim e^{(1-\alpha - \beta)\beta^{-1}\tau}\left(1 + \tau\right) + e^{(1-\beta - \gamma\beta)\beta^{-1}\tau}.  
\end{align*}
Since $1 - \alpha - \beta < 0$, we may conclude also in this case that $\vec{v} \in L^\infty_{\tau,\eta}(\Real^+ \times \Real^d)$.  
Therefore Lemma \ref{lem:rescaled_inftybdd} applies with the hypotheses of Theorem \ref{thm:IA2}. Re-writing in terms of $x$ and $t$, this implies \eqref{ineq:LinftyDecay}. A similar proof with Lemma \ref{lem:finite_p_bounded_unifint} in place of Lemma \ref{lem:finite_p_bounded} also proves \textit{(ii)}
\end{proof}

We now prove Theorem \ref{thm:IA2}. 

\begin{proof}(Theorem \ref{thm:IA2}: \textbf{Intermediate Asymptotics II})
To complete the proof of Theorem \ref{thm:IA2}, we estimate the decay of the relative entropy \eqref{def:relativeH}. 
The proof of Theorem \ref{thm:IA} used the estimate \eqref{ineq:velocity_bounded}. Here we use the bound $\norm{\theta(\tau)}_\infty \lesssim 1$ \eqref{ineq:K_Linfty_decay} and \eqref{ineq:KL1loc} to imply, if $\gamma < d$, 
\begin{align}
\norm{e^{d\tau}\grad \K(e^{\tau}\cdot)\ast \theta}_\infty & \leq \left( \norm{ e^{d\tau}\grad \K(e^{\tau}\cdot)\mathbf{1}_{B_1(0)} \ast \theta}_\infty + \norm{ e^{d\tau}\grad \K(e^{\tau}\cdot)\mathbf{1}_{\Real^d \setminus B_1(0)} \ast \theta}_\infty \right) \nonumber \\ 
&  \lesssim \left(1 + e^{(d-\gamma)\tau}\right)\left(\norm{\theta}_\infty  + M\right) \\ & \lesssim 1 + e^{(d-\gamma)\tau} \lesssim e^{(d-\gamma)\tau}. \label{ineq:velocity_bounded_IA2}   
\end{align}
Similarly, if $\gamma = d$ then, for all $\delta > 0$, 
\begin{equation*} 
\norm{e^{d\tau}\grad \K(e^{\tau}\cdot)\ast \theta}_\infty \lesssim 1 + \tau \lesssim_\delta e^{\delta \tau}.
\end{equation*}
The growth of \eqref{ineq:velocity_bounded_IA2} in time is the source of the degraded convergence rate observed in \eqref{ineq:conv_IA2}. As noted above, this is a manifestation of slow decay in the kernel, which causes growth of $e^{d\tau}\grad \K(e^{\tau}\cdot)$ in $L^1_{loc}$. 
Indeed, computing the decay of the relative entropy (with linear or nonlinear diffusion) as above with \eqref{ineq:velocity_bounded_IA2},
\begin{align*}
\frac{d}{d\tau}H(\theta(\tau)|\theta_M) & = \leq -I(\theta) + e^{(1-\alpha-\beta)\beta^{-1}\tau}I(\theta)^{1/2} \left(\int \theta \abs{e^{d\tau}\grad\K(e^{\tau}\cdot)\ast\theta}^2 d\eta \right)^{1/2} \nonumber \\
& \leq (1-e^{-2\delta\tau})I(\theta) + Ce^{ (2(1-\alpha-\beta)\beta^{-1} + 2(d-\gamma) + 2\delta)\tau}. 
\end{align*}
As before, Theorems \ref{thm:rel_entropy} and \ref{thm:CK} imply, 
\begin{equation*}
\norm{\theta(\tau) - \theta_M}_1 \lesssim e^{-\tau\min\left(1,1 + \gamma - \beta^{-1} - \delta\right)}.
\end{equation*}
Re-writing in terms of $x$ and $t$ and interpolating against \eqref{ineq:LinftyDecay} completes the proof. The corresponding argument follows also for $\gamma = d$, absorbing the mild growth of $\norm{e^{d\tau}\K(e^{\tau}\cdot)\ast\theta}_\infty$ into the $\delta$ already introduced.   
\end{proof}

\section{Acknowledgments} 
The author would like to thank Andrea Bertozzi, Thomas Laurent and Nancy Rodr\'iguez for helpful discussions and guidance, 
and to Inwon Kim for helpful discussions as well suggesting the problem.
This work was in part supported by NSF grant DMS-0907931.

\section{Appendix}
\begin{lemma}[Homogeneous Gagliardo-Nirenberg-Sobolev] \label{lem:GNS}
Let $d \geq 2$ and $f:\Real^d \rightarrow \Real$ satisfy $f \in L^p\cap L^q$ and $\grad f^k \in L^r$. Moreover let $1 \leq p \leq rk \leq dk$, $k < q < rkd/(d-r)$ and
\begin{equation}
\frac{1}{r} - \frac{k}{q} - \frac{s}{d} < 0. \label{cond:GNS}
\end{equation}
Then there exists a constant $C_{GNS}$ which depends on $s,p,q,r,d$ such that
\begin{equation}
\norm{f}_{L^q} \leq C_{GNS}\norm{f}^{\alpha_2}_{L^p} \norm{f^k}^{\alpha_1}_{\dot{W}^{s,r}}, \label{eq:GNS}
\end{equation}
where $0 < \alpha_i$ satisfy
\begin{equation}
1 = \alpha_1 k + \alpha_2,
\end{equation}
and
\begin{equation}
\frac{1}{q} - \frac{1}{p} = \alpha_1(\frac{-s}{d} + \frac{1}{r} - \frac{k}{p}).
\end{equation}
\end{lemma}

The following lemma verifies that the distributions defined by the second derivatives 
of admissible kernels behave as expected under mass-invariant scalings. 
\begin{lemma} \label{lem:CZ_rescale}
Let $\K$ be admissible. Then $\forall \, p, \; 1 < p < \infty$, $u \in L^p$ and $t > 0$, we have
\begin{equation}
\norm{\grad\left(t^d\grad\K(t \cdot)\ast u \right)}_p \lesssim_p t \norm{u}_p. \label{ineq:CZ_rescale}
\end{equation}
\end{lemma}
\begin{proof}
We take the second derivative in the sense of distributions.
Let $\phi \in C_c^\infty$, then by the dominated convergence theorem,
\begin{align*}
\int t^d\partial_{x_i}\K(tx) \partial_{x_j} \phi(x) dx & = \lim_{\epsilon \rightarrow 0} \int_{\abs{x} \geq \epsilon}t^d \partial_{x_i}\K(tx) \partial_{x_j}\phi(x) dx \\
& = -t\lim_{\epsilon \rightarrow 0}\int_{\abs{x} = \epsilon} t^{d-1}\partial_{x_y}\K(tx)\frac{x_j}{\abs{x}}\phi(x) dS - t\textup{PV} \int t^{d}\partial_{x_i,x_j}\K(tx) \phi(x) dx.  
\end{align*}
By $\grad \K \in \wkspace{d/(d-1)}$, we have $\grad \K = \mathcal{O}(\abs{x}^{1-d})$ as $x \rightarrow 0$. Therefore for $\epsilon$ sufficiently small, there exists $C > 0$ such that,
\begin{align*}
\abs{t\int_{\abs{x} = \epsilon}t^{d-1}\partial_{x_i}\K(tx)\frac{x_j}{\abs{x}}\phi(x) dS} & \leq Ct{\int_{\abs{x} = \epsilon} \abs{x}^{1-d}\abs{\phi(x)}dS} \\
& = Ct\int_{\abs{x}= 1}\abs{\epsilon x}^{1-d}\abs{\phi(\epsilon x)} \epsilon^{d-1}dS = Ct\abs{\phi(0)}. 
\end{align*}
The admissibility conditions \textbf{(BD)} and \textbf{(KN)} are sufficient to apply the Calder\'{o}n-Zygmund theory \cite{BigStein}, which implies that the principal value integral in the second term is a bounded linear operator on $L^p$ for all $1 < p < \infty$. The operator norms, which are the implicit constants in \eqref{ineq:CZ_rescale}, only depend on the bound in \textbf{(BD)} and on the condition
\begin{equation*}
\int_{\abs{x} > 2\abs{y}} \abs{K(x-y) - K(x)} dx \leq B, 
\end{equation*}
which is implied by \textbf{(BD)} \cite{BigStein}.
Both of these conditions are clearly left invariant under the rescaling in \eqref{ineq:CZ_rescale} and this concludes the proof.  
\end{proof}

\vfill\eject
\bibliographystyle{plain}
\bibliography{nonlocal_eqns,dispersive}

\end{document}